\documentclass{amsart}
\usepackage{amsmath}
\usepackage{amssymb}
\usepackage{amsfonts}

\newtheorem{theorem}{Theorem}[section]
\theoremstyle{plain}

\newtheorem{corollary}{Corollary}[section]

\newtheorem{lemma}{Lemma}[section]

\numberwithin{equation}{section}

\begin{document}

 \title[Some companions of perturbed Ostrowski type inequalities]{Some companions of perturbed Ostrowski type  inequalities for functions whose second derivatives are bounded and applications}
\author[W. J. Liu]{Wenjun Liu}
\address[W. J. Liu]{College of Mathematics and Statistics\\
Nanjing University of Information Science and Technology \\
Nanjing 210044, China}
\email{wjliu@nuist.edu.cn}

\author[Y. T. Zhu]{Yiting Zhu}
\address[Y. T. Zhu]{College of Mathematics and Statistics\\
Nanjing University of Information Science and Technology \\
Nanjing 210044, China}

%\author[J. P. Zhao]{Junping Zhao}
%\address[J. P. Zhao]{College of Science\\
%Xi'an University of Architecture \& Technology\\
%Xi'an 710055, China}


\begin{abstract}
In this paper we establish some companions of perturbed Ostrowski type integral inequalities for functions whose second derivatives are bounded. Some applications to composite quadrature rules, and to probability density functions are also given.
\end{abstract}

\maketitle

\section{Introduction}
In 1938, Ostrowski \cite{o1938} established the following interesting integral
inequality  for differentiable mappings with bounded
derivatives:
\begin{theorem}\label{Th1.1}
Let $f:[a,b]\rightarrow\mathbb{R}$ be a differentiable mapping on
$(a,b)$ whose derivative is bounded on $(a,b)$ and denote
$\|f'\|_{\infty}=\displaystyle{\sup_{t\in(a,b)}}|f'(t)|<\infty$.
Then for all $x\in[a,b]$ we have
\begin{equation}\label{1.1}
\left|f(x)-\frac{1}{b-a}\int_{a}^{b}f(t)dt\right|\leq\left[\frac{1}{4}
+\frac{(x-\frac{a+b}{2})^{2}}{(b-a)^{2}}\right](b-a)\|f'\|_{\infty}.
\end{equation}
The constant $\frac{1}{4}$ is sharp in the sense that it can not be
replaced by a smaller one.
\end{theorem}

In \cite{gs2002}, Guessab and Schmeisser proved  the following companion of
Ostrowski's inequality:
\begin{theorem}\label{Th1.2}
Let $f:[a,b]\rightarrow\mathbb{R}$ be satisfying the Lipschitz condition, i.e., $|f(t)-f(s)|\le M |t-s|$.
Then for all $x\in[a,\frac{a+b}{2}]$ we have
\begin{equation}\label{1.2}
\left|\frac{f(x)+f(a+b-x)}{2}-\frac{1}{b-a}\int_{a}^{b}f(t)dt\right|\leq\left[\frac{1}{8}+2\left(\frac{x-\frac{3a+b}{4}}{b-a}\right)^{2}\right]
(b-a) M.
\end{equation}
The constant $\frac{1}{8}$ is sharp in the sense that it can not be
replaced by a smaller one.  In \eqref{1.2}, the  point
 $x=\frac{3a+b}{4}$ gives the best estimator. % and yields the
%trapezoid type inequality, i.e.,
%\begin{align}
%\left|\frac{f\left(\frac{3a+b}{4}\right)+f\left(\frac{a+3b}{4}\right)}{2}
%-\frac{1}{b-a}\int_{a}^{b}f(t)dt\right|\leq\frac{b-a }{8}M.\label{1.3}
%\end{align}
%The constant $\frac{1}{8}$ in \eqref{1.3} is sharp in the sense mentioned above.
\end{theorem}

Motivated by \cite{gs2002}, Dragomir \cite{d2005} proved some companions of Ostrowski's
inequality, as follows:
\begin{theorem}\label{Th1.3}
Let $f:[a,b]\rightarrow\mathbb{R}$ be an absolutely continuous
mapping on $[a,b]$. Then the following inequalities
\begin{align*}
\left|\frac{f(x)+f(a+b-x)}{2}-\frac{1}{b-a}\int_{a}^{b}f(t)dt\right|
\leq\left\{
{\begin{array}{l}\left[\frac{1}{8}+2\left(\frac{x-\frac{3a+b}{4}}{b-a}\right)^{2}\right]
(b-a)\|f'\|_{\infty},\quad f'\in L^{\infty}[a,b], \\
\frac{2^{1/q}}{(q+1)^{1/q}}\left[\left(\frac{x-a}{b-a}\right)^{q+1}+
\left(\frac{\frac{a+b}{2}-x}{b-a}\right)^{q+1}\right]^{1/q}(b-a)^{1/q}\|f'\|_{p},\\
\quad \quad \quad\quad\quad
p>1, \frac{1}{p}+\frac{1}{q}=1\quad \text{and}\quad f'\in L^{p}[a,b], \\
\left[\frac{1}{4}+\left|\frac{x-\frac{3a+b}{4}}{b-a}\right|\right]\|f'\|_{1},
\quad\quad\quad \quad\quad\quad f'\in L^{1}[a,b] \\
\end{array}}\right.
\end{align*}
hold for all $x\in[a,\frac{a+b}{2}]$.
\end{theorem}

Recently, Alomari \cite{a20111} studied the companion of Ostrowski inequality \eqref{1.2} for differentiable bounded mappings.
In \cite{liu2009}, Liu established some companions of an Ostrowski type integral inequality for functions whose first derivatives are absolutely continuous and second derivatives belong to $L^{p}$ $(1\leq p \leq\infty)$ spaces.

\begin{theorem}\label{Th1.4}
Let $f:[a,b]\rightarrow\mathbb{R}$ be such that $f'$ is absolutely continuous on $[a,b]$ and $f''\in L^{\infty}[a,b]$. Then for all $x\in[a,\frac{a+b}{2}]$ we have
\begin{align}
&\left|\frac{f(x)+f(a+b-x)}{2}-\left(x-\frac{3a+b}{4}\right)\frac{f'(x)-f'(a+b-x)}{2}-\frac{1}{b-a}\int_{a}^{b}f(t)dt\right|\nonumber\\
\leq&\left[\frac{1}{96}+\frac{1}{2}\frac{\left(x-\frac{3a+b}{4}\right)^2}{(b-a)^2}\right](b-a)^2\|f''\|_{\infty}. \label{1.5'}
\end{align}
The constant $\frac{1}{96}$ is sharp in the sense that it can not be
replaced by a smaller one.
\end{theorem}

For other related results, the reader may be refer to \cite{a20112, a20113, bdg2009, d20051,  d2011, d2001, hn2011, l2008, l2010, lxw2010, l2007, s20101, s2010, ss2011, thd2008, thyc2011, u2003, v2011} and the
references therein.

The main aim of this paper is to establish  some companions of perturbed Ostrowski type integral inequalities for functions whose second derivatives are bounded (Theorem \ref{th2.1}-\ref{th2.5}). Some applications to composite quadrature rules, and to probability density functions are also given.

\section{Main results}
To prove our main results, we need the following lemmas.
\begin{lemma}\label{lem2.1}
\cite{liu2009} Let $f:[a,b]\rightarrow\mathbb{R}$ be such that $f'$ is absolutely continuous on $[a,b]$. Denote by $K(x,t):[a,b]\rightarrow\mathbb{R}$
the kernel given by
\begin{align}
K(x,t)=\left\{{\begin{aligned} &\frac{1}{2}(t-a)^2,&&t\in[a,x],\\
&\frac{1}{2}\left(t-\frac{a+b}{2}\right)^2,&&t\in(x,a+b-x],\\
&\frac{1}{2}(t-b)^2,&&t\in(a+b-x,b],\\
\end{aligned}}\right.\label{2.1}
\end{align}
then the identity
\begin{align}
&\frac{1}{b-a}\int_{a}^{b}K(x,t)f''(t)dt\nonumber\\
=&\frac{1}{b-a}\int_{a}^{b}f(t)dt-\frac{f(x)+f(a+b-x)}{2}+\left(x-\frac{3a+b}{4}\right)\frac{f'(x)-f'(a+b-x)}{2}\label{2.2}
\end{align}
holds.
\end{lemma}

\begin{lemma}\label{lem2.2}
\cite[Gr\"uss inequality]{g1935} Let $f,g:[a,b]\rightarrow\mathbb{R}$ be two integrable functions such that $\phi\leq f(t)\leq \Phi$ and $\gamma\leq g(t)\leq\Gamma$ for all $t\in[a,b],$ where $\phi,\Phi,\gamma$ and $\Gamma$ are constants. Then we have
\begin{align}\label{2.3}
 \left|\frac{1}{b-a}\int_{a}^{b}f(t)g(t)dt-\frac{1}{b-a}\int_{a}^{b}f(t)dt\cdot\frac{1}{b-a}\int_{a}^{b}g(t)dt\right|
\leq \frac{1}{4}(\Phi-\phi)(\Gamma-\gamma).
\end{align}
\end{lemma}

\begin{theorem}\label{th2.1}
Let $f:[a,b]\rightarrow\mathbb{R}$ be such that $f'$ is absolutely continuous on $[a,b]$. If $f''\in L^1[a,b]$ and $\gamma\leq f''(x)\leq\Gamma$, $\forall\ x\in[a,b]$, then for all $x\in\left[a,\frac{a+b}{2}\right]$ we have
\begin{align}\label{2.4}
&\left|\frac{f(x)+f(a+b-x)}{2}-\left(x-\frac{3a+b}{4}\right)\frac{f'(x)-f'(a+b-x)}{2}\right.\nonumber\\
&\left.+\frac{f'(b)-f'(a)}{b-a}\left[\frac{1}{2}\left(x-\frac{3a+b}{4}\right)^2+\frac{(b-a)^2}{96}\right]-\frac{1}{b-a}\int_{a}^{b}f(t)dt\right|\nonumber\\
\leq &\frac{1}{8}(\Gamma-\gamma)\left[\frac{b-a}{4}+\left|x-\frac{3a+b}{4}\right|\right]^2.
\end{align}
\end{theorem}

\begin{proof}
It is clear that for all $t\in[a,b]$ and $x\in\left[a,\frac{a+b}{2}\right]$, we have
\begin{align*}
0\leq &K(x,t)\leq\max\left\{\frac{1}{2}(x-a)^2,\ \frac{1}{2}\left(\frac{a+b}{2}-x\right)^2\right\}\\
=&\frac{1}{4}\left\{\left[(x-a)^2+\left(\frac{a+b}{2}-x\right)^2\right]+\left|(x-a)^2-\left(\frac{a+b}{2}-x\right)^2\right|\right\}\\
=&\frac{1}{2}\left[\left(x-\frac{3a+b}{4}\right)^2+\frac{(b-a)^2}{16}+\frac{b-a}{2}\left|x-\frac{3a+b}{4}\right|\right]\\
=&\frac{1}{2}\left[\frac{b-a}{4}+\left|x-\frac{3a+b}{4}\right|\right]^2.
\end{align*}
Applying Lemma \ref{lem2.2} to the functions $K(x,\cdot)$ and $f''(\cdot)$, we get
\begin{align}\label{2.5}
&\left|\frac{1}{b-a}\int_{a}^{b}K(x,t)f''(t)dt-\frac{1}{b-a}\int_{a}^{b}K(x,t)dt\cdot\frac{1}{b-a}\int_{a}^{b}f''(t)dt\right|\nonumber\\
\leq &\frac{1}{8}(\Gamma-\gamma)\left[\frac{b-a}{4}+\left|x-\frac{3a+b}{4}\right|\right]^2
\end{align}
for all $x\in[a,\frac{a+b}{2}]$. By a simple calculation, we obtain
\begin{equation}\label{2.6}
\frac{1}{b-a}\int_{a}^{b}f''(t)dt=\frac{f'(b)-f'(a)}{b-a}
\end{equation}
and
\begin{equation}\label{2.7}
\frac{1}{b-a}\int_{a}^{b}K(x,t)dt=\frac{1}{2}\left(x-\frac{3a+b}{4}\right)^2+\frac{(b-a)^2}{96}.
\end{equation}
Combining  \eqref{2.2},  \eqref{2.5}-\eqref{2.7}, we obtain \eqref{2.4} as required.
\end{proof}

\begin{corollary}
In the inequality \eqref{2.4}, choose

$(1)$ $x=\frac{3a+b}{4}$, we get
\begin{align}\label{2.8'}
 \left|\frac{f(\frac{3a+b}{4})+f(\frac{a+3b}{4})}{2}+\frac{f'(b)-f'(a)}{b-a}\frac{(b-a)^2}{96}-\frac{1}{b-a}\int_{a}^{b}f(t)dt\right|
\leq \frac{1}{128}(\Gamma-\gamma)(b-a)^2.
\end{align}

$(2)$ $x=a$, we get
\begin{align*}
 \left|\frac{f(a)+f(b)}{2}-\frac{f'(b)-f'(a)}{b-a}\frac{(b-a)^2}{12}-\frac{1}{b-a}\int_{a}^{b}f(t)dt\right|
\leq \frac{1}{32}(\Gamma-\gamma)(b-a)^2,
\end{align*}
which is better than \cite[Corollary 2.3]{cdr1999}
since a smaller estimator is given here.

$(3)$ $x=\frac{a+b}{2}$, we get
\begin{align*}
 \left|f\left(\frac{a+b}{2}\right)+\frac{f'(b)-f'(a)}{b-a}\frac{(b-a)^2}{24}-\frac{1}{b-a}\int_{a}^{b}f(t)dt\right|
\leq \frac{1}{32}(\Gamma-\gamma)(b-a)^2,
\end{align*}
which is the inequality given in \cite[Corollay 2.2]{cdr1999}.
\end{corollary}

\begin{corollary}
Let $f$ as in Theorem \ref{th2.1}. Additionally, if $f$ is symmetric about $x=\frac{a+b}{2}$, i.e., $f(a+b-x)=f(x)$, then we have
\begin{align*}
 &\left|f(x)-\left(x-\frac{3a+b}{4}\right)f'(x)\right. \left.+\frac{f'(b)-f'(a)}{b-a}\left[\frac{1}{2}\left(x-\frac{3a+b}{4}\right)^2+\frac{(b-a)^2}{96}\right]-\frac{1}{b-a}\int_{a}^{b}f(t)dt\right|\nonumber\\
\leq&\frac{1}{8}(\Gamma-\gamma)\left[\frac{b-a}{4}+\left|x-\frac{3a+b}{4}\right|\right]^2
\end{align*}
for all $x\in\left[a,\frac{a+b}{2}\right]$.
\end{corollary}

\begin{theorem}\label{th2.2}
Let $f:[a,b]\rightarrow\mathbb{R}$ be such that $f'$ is absolutely continuous on $[a,b]$. If $f''\in L^1[a,b]$ and $\gamma\leq f''(x)\leq\Gamma$, $\forall\ x\in[a,b]$, then for all $x\in\left[a,\frac{a+b}{2}\right]$ we have
\begin{align}\label{2.8}
&\left|\frac{f(x)+f(a+b-x)}{2}-\left(x-\frac{3a+b}{4}\right)\frac{f'(x)-f'(a+b-x)}{2}\right.\nonumber\\
&\left.+\frac{\Gamma+\gamma}{2}\left[\frac{1}{2}\left(x-\frac{3a+b}{4}\right)^2+\frac{(b-a)^2}{96}\right]
-\frac{1}{b-a}\int_{a}^{b}f(t)dt\right|\nonumber\\
\leq& \frac{\Gamma-\gamma}{2}\left[\frac{1}{2}\left(x-\frac{3a+b}{4}\right)^2+\frac{(b-a)^2}{96}\right].
\end{align}
\end{theorem}

\begin{proof}
From \eqref{2.2} and \eqref{2.7}, we have
\begin{align*}
&\frac{1}{b-a}\int_{a}^{b}K(x,t)[f''(t)-C]dt\nonumber\\=&\frac{1}{b-a}\int_{a}^{b}f(t)dt-\frac{f(x)+f(a+b-x)}{2}\nonumber\\
&+\left(x-\frac{3a+b}{4}\right)\frac{f'(x)-f'(a+b-x)}{2}-C\left[\frac{1}{2}\left(x-\frac{3a+b}{4}\right)^2+\frac{(b-a)^2}{96}\right].
\end{align*}
Let $C=\frac{\Gamma+\gamma}{2}$, we get
\begin{align}\label{2.9}
&\left|\frac{f(x)+f(a+b-x)}{2}-\left(x-\frac{3a+b}{4}\right)\frac{f'(x-f'(a+b-x))}{2}\right.\nonumber\\
&\left.+\frac{\Gamma+\gamma}{2}\left[\frac{1}{2}\left(x-\frac{3a+b}{4}\right)^2+\frac{(b-a)^2}{96}\right]
-\frac{1}{b-a}\int_{a}^{b}f(t)dt\right|\nonumber\\
\leq &\max\limits_{t\in [a,b]}|f''(t)-C|\frac{1}{b-a}\int_{a}^{b}|K(x,t)|dt.
\end{align}
We also have
\begin{equation}\label{2.10}
\max\limits_{t\in [a,b]}|f''(t)-C|\leq\frac{\Gamma-\gamma}{2}
\end{equation}
and
\begin{equation}\label{2.11}
\frac{1}{b-a}\int_{a}^{b}|K(x,t)|dt=\frac{1}{2}\left(x-\frac{3a+b}{4}\right)^2+\frac{(b-a)^2}{96}.
\end{equation}
Therefore, from \eqref{2.9}-\eqref{2.11}, we obtain the desired inequality \eqref{2.8}.
\end{proof}

\begin{corollary}
In the inequality \eqref{2.8}, choose

$(1)$ $x=\frac{3a+b}{4}$, we get
\begin{equation}\label{2.12'}
\left|\frac{f(\frac{3a+b}{4})+f(\frac{a+3b}{4})}{2}
+\frac{\Gamma+\gamma}{2}\frac{(b-a)^2}{96}-\frac{1}{b-a}\int_{a}^{b}f(t)dt\right|
\leq \frac{1}{192}(\Gamma-\gamma)(b-a)^2.
\end{equation}

$(2)$ $x=a$, we get
\begin{align*}
 \left|\frac{f(a)+f(b)}{2}-\frac{f'(b)-f'(a)}{b-a}\frac{(b-a)^2}{8}+\frac{\Gamma+\gamma}{2}\frac{(b-a)^2}{24}-\frac{1}{b-a}\int_{a}^{b}f(t)dt\right|
\leq  \frac{1}{48}(\Gamma-\gamma)(b-a)^2.
\end{align*}

$(3)$ $x=\frac{a+b}{2}$, we get
\begin{equation*}
\left|f\left(\frac{a+b}{2}\right)+\frac{\Gamma+\gamma}{2}\frac{(b-a)^2}{24}-\frac{1}{b-a}\int_{a}^{b}f(t)dt\right|
\leq \frac{1}{48}(\Gamma-\gamma)(b-a)^2.
\end{equation*}
\end{corollary}

\begin{corollary}
Let $f$ as in Theorem \ref{th2.2}. Additionally, if $f$ is symmetric about $x=\frac{a+b}{2}$, then we have
\begin{align*}
&\left|f(x)-\left(x-\frac{3a+b}{4}\right)f'(x)\right. \left.+\frac{\Gamma+\gamma}{2}\left[\frac{1}{2}\left(x-\frac{3a+b}{4}\right)^2+\frac{(b-a)^2}{96}\right]
-\frac{1}{b-a}\int_{a}^{b}f(t)dt\right|\nonumber\\
\leq& \frac{\Gamma-\gamma}{2}\left[\frac{1}{2}\left(x-\frac{3a+b}{4}\right)^2+\frac{(b-a)^2}{96}\right],
\end{align*}
for all $x\in\left[a,\frac{a+b}{2}\right]$.
\end{corollary}

\begin{theorem}\label{th2.3}
Let $f:[a,b]\rightarrow\mathbb{R}$ be such that $f'$ is absolutely continuous on $[a,b]$.
If $f''\in L^1[a,b]$ and $\gamma\leq f''(x)\leq\Gamma$, $\forall\ x\in[a,b]$, then for all $x\in\left[a,\frac{a+b}{2}\right]$ we have
\begin{align}\label{2.14}
&\left|\frac{f(x)+f(a+b-x)}{2}-\left(x-\frac{3a+b}{4}\right)\frac{f'(x)-f'(a+b-x)}{2}\right.\nonumber\\
&\left.+\frac{f'(b)-f'(a)}{b-a}\left[\frac{1}{2}\left(x-\frac{3a+b}{4}\right)^2+\frac{(b-a)^2}{96}\right]-\frac{1}{b-a}\int_{a}^{b}f(t)dt\right|\nonumber\\
\leq &(S-\gamma)\left[\frac{(b-a)^2}{48}+ \frac{b-a}{4}\left|x-\frac{3a+b}{4}\right|\right]
\end{align}
and
\begin{align}\label{2.15}
&\left|\frac{f(x)+f(a+b-x)}{2}-\left(x-\frac{3a+b}{4}\right)\frac{f'(x)-f'(a+b-x)}{2}\right.\nonumber\\
&\left.+\frac{f'(b)-f'(a)}{b-a}\left[\frac{1}{2}\left(x-\frac{3a+b}{4}\right)^2+\frac{(b-a)^2}{96}\right]-\frac{1}{b-a}\int_{a}^{b}f(t)dt\right|\nonumber\\
\leq & (\Gamma-S)\left[\frac{(b-a)^2}{48}+ \frac{b-a}{4}\left|x-\frac{3a+b}{4}\right|\right],
\end{align}
where $S=(f'(b)-f'(a))/(b-a)$.
%If $\gamma,\Gamma$ are given by
%\begin{equation}\label{2.16}
%\gamma=\inf\limits_{t\in [a,b]}f''(t),\Gamma=\sup\limits_{t\in[a,b]}f''(t),
%\end{equation}
%then the constant $\frac{1}{2}$ in \eqref{2.14} and \eqref{2.15} is sharp in the sense that it can not be replaced by a smaller one.
\end{theorem}
\begin{proof}
From \eqref{2.2}, \eqref{2.6}, \eqref{2.7}, it follows that
\begin{align}\label{2.17}
&\frac{1}{b-a}\int_{a}^{b}K(x,t)f''(t)dt-\frac{1}{(b-a)^2}\int_{a}^{b}K(x,t)dt\int_{a}^{b}f''(t)dt\nonumber\\
=&\frac{1}{b-a}\int_{a}^{b}f(t)dt-\frac{f(x)+f(a+b-x)}{2}+\left(x-\frac{3a+b}{4}\right)\frac{f'(x)-f'(a+b-x)}{2}\nonumber\\
&-\frac{f'(b)-f'(a)}{b-a}\left[\frac{1}{2}\left(x-\frac{3a+b}{4}\right)^2+\frac{(b-a)^2}{96}\right].
\end{align}
We denote
\begin{equation}\label{2.18}
R_{n}(x)=\frac{1}{b-a}\int_{a}^{b}K(x,t)f''(t)dt-\frac{1}{(b-a)^2}\int_{a}^{b}K(x,t)dt\int_{a}^{b}f''(t)dt.
\end{equation}
If $C\in\mathbb{R}$ is an arbitrary constant, then we have
\begin{equation}\label{2.19}
R_{n}(x)=\frac{1}{b-a}\int_{a}^{b}(f''(t)-C)\left[K(x,t)-\frac{1}{b-a}\int_{a}^{b}K(x,s)ds\right]dt.
\end{equation}
Furthermore, we have
\begin{align}\label{2.20}
|R_{n}(x)|\leq\frac{1}{b-a}\max\limits_{t\in[a,b]}\left|K(x,t)-\frac{1}{b-a}\int_{a}^{b}K(x,s)ds\right|\int_{a}^{b}|f''(t)-C|dt.
\end{align}

To compute
\begin{align}\label{2.21}
&\max\limits_{t\in[a,b]}\left|K(x,t)-\frac{1}{b-a}\int_{a}^{b}K(x,s)ds\right|\nonumber\\
=&\max\left\{\left|\frac{1}{2}(x-a)^2-\left[\frac{1}{2}\left(x-\frac{3a+b}{4}\right)^2+\frac{(b-a)^2}{96}\right]\right|,\right.\nonumber\\
&\left.\left|\frac{1}{2}\left(\frac{a+b}{2}-x\right)^2-\left[\frac{1}{2}\left(x-\frac{3a+b}{4}\right)^2+\frac{(b-a)^2}{96}\right]\right|, \left[\frac{1}{2}\left(x-\frac{3a+b}{4}\right)^2+\frac{(b-a)^2}{96}\right]\right\}\nonumber\\
=&\max\left\{\frac{b-a}{24}|6x-5a-b|,  \frac{b-a}{12}|3x-2a-b|, \left[\frac{1}{2}\left(x-\frac{3a+b}{4}\right)^2+\frac{(b-a)^2}{96}\right]\right\},
\end{align}
we  denote
\begin{align*}
y_{1}=\frac{b-a}{24}|6x-5a-b|,\
y_{2}=\frac{b-a}{12}|3x-2a-b|,\
y_{3}=\frac{1}{2}\left(x-\frac{3a+b}{4}\right)^2+\frac{(b-a)^2}{96}.
\end{align*}
%If we choose $y_{1}=0$, then we get $x_{1}=\frac{5a+b}{6}$. If we choose $y_{2}=0$, then we get $x_{2}=\frac{2a+b}{3}$.
%
%$(1)$ If $x\in[a,x_{1}]$, we have
%\begin{align}
%y_{1}-y_{2}&=\frac{b-a}{24}(-6x+5a+b)-\frac{b-a}{12}(-3x+2a+b)\nonumber\\
%&=-\frac{(b-a)^2}{24}<0
%\end{align}
%Therefore, we get $y_{1}<y_{2}$.
%
%$(2)$ If $x\in\left(x_{1},\frac{3a+b}{4}\right]$, we have
%\begin{align}
%y_{1}-y_{2}&=\frac{b-a}{24}(6x-5a-b)-\frac{b-a}{12}(-3x+2a+b)\nonumber\\
%&=\frac{b-a}{2}\left(x-\frac{3a+b}{4}\right)
%\end{align}
%since $x\in\left(x_{1},\frac{3a+b}{4}\right]$, then we get
%$y_{1}-y_{2}\leq0$.
%Therefore, we get $y_{1}\leq y_{2}$.
%
%$(3)$ If $x\in\left(\frac{3a+b}{4},x_{2}\right]$, we have
%\begin{align}
%y_{1}-y_{2}&=\frac{b-a}{24}(6x-5a-b)-\frac{b-a}{12}(-3x+2a+b)\nonumber\\
%&=\frac{b-a}{2}\left(x-\frac{3a+b}{4}\right)
%\end{align}
%since $x\in\left(\frac{3a+b}{4},x_{2}\right]$, then we get
%$y_{1}-y_{2}>0$.
%Therefore, we get $y_{1}>y_{2}$.
%
%$(4)$ If $x\in\left(x_{2},\frac{a+b}{2}\right]$, we have
%\begin{align}
%y_{1}-y_{2}&=\frac{b-a}{24}(6x-5a-b)-\frac{b-a}{12}(3x-2a-b)\nonumber\\
%&=\frac{(b-a)^2}{24}>0
%\end{align}
%Therefore, we get $y_{1}<y_{2}$.
%
%According to the above relations, we get
A direct computation gives that
\begin{align}
\left\{{\begin{aligned}&y_{2}\geq \max\{y_{1}, y_3\}, &&x\in\left[a,\frac{3a+b}{4}\right],\\
&y_{1}>\max\{y_{2}, y_3\}, &&x\in\left(\frac{3a+b}{4},\frac{a+b}{2}\right].\\
\end{aligned}}\right.
\end{align}
%Furthermore, let's compare $y_{2}$ with $y_{3}$ in $x\in\left[a,\frac{3a+b}{4}\right]$, then we get
%\begin{align}
%y_{2}-y_{3}=&\frac{b-a}{12}(-3x+2a+b)-\left[\frac{1}{2}\left(x-\frac{3a+b}{4}\right)^2+\frac{(b-a)^2}{96}\right]\nonumber\\
%=&-\frac{1}{2}(x-a)^2+\frac{(b-a)^2}{24},
%\end{align}
%since $x\in\left[a,\frac{3a+b}{4}\right]$, then we get
%\begin{align}
%-\frac{1}{2}(x-a)^2+\frac{(b-a)^2}{24}\geq\frac{(b-a)^2}{96}>0.
%\end{align}
%Therefore, we get $y_{2}>y_{3}$. \\
%In a similar way, we get
%\begin{align}
%y_{1}-y_{3}=-\frac{1}{2}\left(x-\frac{a+b}{2}\right)^2+\frac{(b-a)^2}{24}\geq\frac{(b-a)^2}{96}>0,
%\end{align}
%therefore, we have $y_{1}>y_{3},x\in\left(\frac{3a+b}{4},\frac{a+b}{2}\right]$. \\
Therefore, we get
\begin{align}\label{2.34}
\max\limits_{t\in[a,b]}\left|K(x,t)-\frac{1}{b-a}\int_{a}^{b}K(x,s)ds\right|
=\max\left\{y_{1},y_{2}\right\}=\left[\frac{(b-a)^2}{48}+ \frac{b-a}{4}\left|x-\frac{3a+b}{4}\right|\right].
\end{align}

We also have
\begin{equation}\label{2.22}
\int_{a}^{b}|f''(t)-\gamma|dt=(S-\gamma)(b-a)
\end{equation}
and
\begin{equation}\label{2.23'}
\int_{a}^{b}|f''(t)-\Gamma|dt=(\Gamma-S)(b-a).
\end{equation}
Therefore, we obtain \eqref{2.14} and \eqref{2.15} by using $\eqref{2.17}-\eqref{2.20}$, $\eqref{2.34}-\eqref{2.23'}$ and choosing $C=\gamma$ and $C=\Gamma$ in $\eqref{2.20}$, respectively.
\end{proof}
\begin{corollary}
Under the assumptions of Theorem \ref{th2.3}, choose

$(1)$ $x=\frac{3a+b}{4}$, we have
\begin{align}\label{2.23}
\left|\frac{f(\frac{3a+b}{4})+f(\frac{a+3b}{4})}{2}+\frac{f'(b)-f'(a)}{b-a}\frac{(b-a)^2}{96}-\frac{1}{b-a}\int_{a}^{b}f(t)dt\right|
\leq \frac{1}{48}(S-\gamma)(b-a)^2,
\end{align}
\begin{align}\label{2.24}
\left|\frac{f(\frac{3a+b}{4})+f(\frac{a+3b}{4})}{2}+\frac{f'(b)-f'(a)}{b-a}\frac{(b-a)^2}{96}-\frac{1}{b-a}\int_{a}^{b}f(t)dt\right|
\leq \frac{1}{48}(\Gamma-S)(b-a)^2.
\end{align}

$(2)$ $x=a$, we have
\begin{align*}
\left|\frac{f(a)+f(b)}{2}-\frac{f'(b)-f'(a)}{b-a}\frac{(b-a)^2}{12}-\frac{1}{b-a}\int_{a}^{b}f(t)dt\right|
\leq \frac{1}{12}(S-\gamma)(b-a)^2,
\end{align*}
\begin{align*}
\left|\frac{f(a)+f(b)}{2}-\frac{f'(b)-f'(a)}{b-a}\frac{(b-a)^2}{12}-\frac{1}{b-a}\int_{a}^{b}f(t)dt\right|
\leq \frac{1}{12}(\Gamma-S)(b-a)^2.
\end{align*}

$(3)$ $x=\frac{a+b}{2}$, we have
\begin{align*}
 \left|f\left(\frac{a+b}{2}\right)+\frac{f'(b)-f'(a)}{b-a}\frac{(b-a)^2}{24}-\frac{1}{b-a}\int_{a}^{b}f(t)dt\right|
\leq \frac{1}{12}(S-\gamma)(b-a)^2,
\end{align*}
\begin{align*}
 \left|f\left(\frac{a+b}{2}\right)+\frac{f'(b)-f'(a)}{b-a}\frac{(b-a)^2}{24}-\frac{1}{b-a}\int_{a}^{b}f(t)dt\right|
\leq \frac{1}{12}(\Gamma-S)(b-a)^3.
\end{align*}
\end{corollary}
\begin{corollary}
Let $f$ as in Theorem \ref{th2.3}. Additionally, if $f$ is symmetric about $x=\frac{a+b}{2}$,   then for all $x\in\left[a,\frac{a+b}{2}\right]$ we have
\begin{align*}
&\left|f(x)-\left(x-\frac{3a+b}{4}\right)f'(x)+\frac{f'(b)-f'(a)}{b-a}\left[\frac{1}{2}\left(x-\frac{3a+b}{4}\right)^2+\frac{(b-a)^2}{96}\right]-\frac{1}{b-a}\int_{a}^{b}f(t)dt\right|\nonumber\\
\leq & (S-\gamma)\left[\frac{(b-a)^2}{48}+ \frac{b-a}{4}\left|x-\frac{3a+b}{4}\right|\right]
\end{align*}
and
\begin{align*}
&\left|f(x)-\left(x-\frac{3a+b}{4}\right)f'(x)+\frac{f'(b)-f'(a)}{b-a}\left[\frac{1}{2}\left(x-\frac{3a+b}{4}\right)^2+\frac{(b-a)^2}{96}\right]-\frac{1}{b-a}\int_{a}^{b}f(t)dt\right|\nonumber\\
\leq & (\Gamma-S)\left[\frac{(b-a)^2}{48}+ \frac{b-a}{4}\left|x-\frac{3a+b}{4}\right|\right].
\end{align*}
\end{corollary}
\begin{theorem}\label{th2.4}
Let $f:[a,b]\rightarrow\mathbb{R}$ be a thrice continuously differentiable mapping in $(a,b)$ with $f'''\in L^{2}[a,b]$. Then for all $x\in\left[a,\frac{a+b}{2}\right]$ we have

\begin{align}\label{2.29}
&\left|\frac{f(x)+f(a+b-x)}{2}-\left(x-\frac{3a+b}{4}\right)\frac{f'(x)-f'(a+b-x)}{2}\right.\nonumber\\
&\left.+\frac{f'(b)-f'(a)}{b-a}\left[\frac{1}{2}\left(x-\frac{3a+b}{4}\right)^2+\frac{(b-a)^2}{96}\right]-\frac{1}{b-a}\int_{a}^{b}f(t)dt\right|\nonumber\\
\leq&\frac{1}{\pi}\|f'''\|_{2}\left\{\frac{1}{320}(a+b-2x)^5+\frac{1}{10}(x-a)^5-(b-a)\left[\frac{1}{2}\left(x-\frac{3a+b}{4}\right)^2+\frac{(b-a)^2}{96}\right]^2\right\}^{1/2}. \end{align}
\end{theorem}
\begin{proof}
Let $R_n(x)$ be defined by \eqref{2.18}. From \eqref{2.17}, we get
\begin{align}
R_{n}(x)=&\frac{1}{b-a}\int_{a}^{b}f(t)dt-\frac{f(x)+f(a+b-x)}{2}+\left(x-\frac{3a+b}{4}\right)\frac{f'(x)-f'(a+b-x)}{2}\nonumber\\
&-\frac{f'(b)-f'(a)}{b-a}\left[\frac{1}{2}\left(x-\frac{3a+b}{4}\right)^2+\frac{(b-a)^2}{96}\right].
\end{align}
If we choose $C=f''((a+b)/2)$ in \eqref{2.19} and use the Cauchy inequality, then we get
\begin{align}
&|R_{n}(x)|\nonumber\\\leq&\frac{1}{b-a}\int_{a}^{b}\left|f''(t)-f''\left(\frac{a+b}{2}\right)\right|\left|K(x,t)-\frac{1}{b-a}\int_{a}^{b}K(x,s)ds\right|dt\nonumber\\
\leq&\frac{1}{b-a}\left[\int_{a}^{b}\left(f''(t)-f''\left(\frac{a+b}{2}\right)\right)^2dt\right]^{1/2}\left[\int_{a}^{b}\left(K(x,t)-\frac{1}{b-a}\int_{a}^{b}K(x,s)ds\right)^{2}dt\right]^{1/2}. \end{align}
We can use the Diaz-Metcalf inequality to get
\begin{align*}
\int_{a}^{b}\left(f''(t)-f''\left(\frac{a+b}{2}\right)\right)^2dt\leq\frac{(b-a)^2}{\pi^2}\|f'''\|_{2}^{2}.
\end{align*}
We also have
\begin{align}\label{2.28'}
&\int_{a}^{b}\left(K(x,t)-\frac{1}{b-a}\int_{a}^{b}K(x,s)ds\right)^{2}dt\nonumber\\=&\int_{a}^{b}K(x,t)^{2}dt-(b-a)\left[\frac{1}{2}\left(x-\frac{3a+b}{4}\right)^2+\frac{(b-a)^2}{96}\right]^2\nonumber\\
=&\frac{1}{320}(a+b-2x)^5+\frac{1}{10}(x-a)^5-(b-a)\left[\frac{1}{2}\left(x-\frac{3a+b}{4}\right)^2+\frac{(b-a)^2}{96}\right]^2.
\end{align}
Therefore, using the above relations, we obtain \eqref{2.29}.
\end{proof}
\begin{corollary}
Under the assumptions of Theorem \ref{th2.4}, choose

$(1)$ $x=\frac{3a+b}{4}$, we have
\begin{align}\label{2.33}
\left|\frac{f(\frac{3a+b}{4})+f(\frac{a+3b}{4})}{2}+\frac{f'(b)-f'(a)}{b-a}\frac{(b-a)^2}{96}-\frac{1}{b-a}\int_{a}^{b}f(t)dt\right|
\leq\frac{(b-a)^{5/2}}{48\pi\sqrt{5}}\|f'''\|_{2}.
\end{align}

$(2)$ $x=a$, we have
\begin{align*}
\left|\frac{f(a)+f(b)}{2}-\frac{f'(b)-f'(a)}{b-a}\frac{(b-a)^2}{12}-\frac{1}{b-a}\int_{a}^{b}f(t)dt\right|
\leq\frac{(b-a)^{5/2}}{12\pi\sqrt{5}}\|f'''\|_{2}.
\end{align*}

$(3)$ $x=\frac{a+b}{2}$, we have
\begin{align*}
 \left|f\left(\frac{a+b}{2}\right)+\frac{f'(b)-f'(a)}{b-a}\frac{(b-a)^2}{24}-\frac{1}{b-a}\int_{a}^{b}f(t)dt\right|
 \leq\frac{(b-a)^{5/2}}{12\pi\sqrt{5}}\|f'''\|_{2}.
 \end{align*}
 %It is obvious that \eqref{2.37} gives a smaller estimator than the above inequality.
\end{corollary}
\begin{corollary}
Let $f$ as in Theorem \ref{th2.3}. Additionally, if $f$ is symmetric about $x=\frac{a+b}{2}$, i.e., $f(a+b-x)=f(x)$, then for all $x\in\left[a,\frac{a+b}{2}\right]$ we have
\begin{align*}
&\left|f(x)-\left(x-\frac{3a+b}{4}\right)f'(x)+\frac{f'(b)-f'(a)}{b-a}\left[\frac{1}{2}\left(x-\frac{3a+b}{4}\right)^2+\frac{(b-a)^2}{96}\right]-\frac{1}{b-a}\int_{a}^{b}f(t)dt\right|\nonumber\\
\leq&\frac{1}{\pi}\|f'''\|_{2}\left\{\frac{1}{320}(a+b-2x)^5+\frac{1}{10}(x-a)^5-(b-a)\left[\frac{1}{2}\left(x-\frac{3a+b}{4}\right)^2+\frac{(b-a)^2}{96}\right]^2\right\}^{1/2}. \end{align*}
\end{corollary}

\begin{theorem}\label{th2.5}
Let $f:[a,b]\rightarrow\mathbb{R}$ be such that $f'$ is absolutely continuous on $[a,b]$ with $f''\in L^{2}[a,b]$. Then for all $x\in\left[a,\frac{a+b}{2}\right]$ we have
\begin{align}\label{2.36}
&\left|\frac{f(x)+f(a+b-x)}{2}-\left(x-\frac{3a+b}{4}\right)\frac{f'(x)-f'(a+b-x)}{2}\right.\nonumber\\
&\left.+\frac{f'(b)-f'(a)}{b-a}\left[\frac{1}{2}\left(x-\frac{3a+b}{4}\right)^2+\frac{(b-a)^2}{96}\right]-\frac{1}{b-a}\int_{a}^{b}f(t)dt\right|\nonumber\\
\leq&\frac{\sqrt{\sigma(f'')}}{b-a}\left\{\frac{1}{320}(a+b-2x)^5+\frac{1}{10}(x-a)^5-(b-a)\left[\frac{1}{2}\left(x-\frac{3a+b}{4}\right)^2+\frac{(b-a)^2}{96}\right]^2\right\}^{1/2},  \end{align}
where $\sigma(f'')$ is defined by
\begin{align}
\sigma(f'')=\|f''\|_{2}^{2}-\frac{(f'(b)-f'(a))^2}{b-a}=\|f''\|_{2}^{2}-S^2(b-a)
\end{align}
and $S$ is defined in Theorem \ref{th2.3}.
\end{theorem}
\begin{proof}
Let $R_{n}(x)$ be defined by \eqref{2.18}.
If we choose $C=\frac{1}{b-a}\int_{a}^{b}f''(s)ds$ in \eqref{2.19} and use the Cauchy inequality and \eqref{2.28'}, then we get
\begin{align*}
&|R_{n}(x)|\nonumber\\\leq &\frac{1}{b-a}\int_{a}^{b}\left|f''(t)-\frac{1}{b-a}\int_{a}^{b}f''(s)ds\right|\left|K(x,t)-\frac{1}{b-a}\int_{a}^{b}K(x,s)ds\right|dt\nonumber\\
\leq &\frac{1}{b-a}\left[\int_{a}^{b}\left(f''(t)-\frac{1}{b-a}\int_{a}^{b}f''(s)ds\right)^2dt\right]^{1/2}\left[\int_{a}^{b}\left(K(x,t)-\frac{1}{b-a}\int_{a}^{b}K(x,s)ds\right)^{2}dt\right]^{1/2}\nonumber\\
\leq&\frac{\sqrt{\sigma(f'')}}{b-a}\left\{\frac{1}{320}(a+b-2x)^5+\frac{1}{10}(x-a)^5-(b-a)\left[\frac{1}{2}\left(x-\frac{3a+b}{4}\right)^2+\frac{(b-a)^2}{96}\right]^2\right\}^{1/2}.
\end{align*}
\end{proof}

\begin{corollary}
Under the assumptions of Theorem \ref{th2.5}, choose

$(1)$ $x=\frac{3a+b}{4}$, we have
\begin{align}\label{2.38}
\left|\frac{f(\frac{3a+b}{4})+f(\frac{a+3b}{4})}{2}+\frac{f'(b)-f'(a)}{b-a}\frac{(b-a)^2}{96}-\frac{1}{b-a}\int_{a}^{b}f(t)dt\right|
\leq\frac{(b-a)^{3/2}}{48\sqrt{5}}\sqrt{\sigma(f'')}.
\end{align}

$(2)$ $x=a$, we have
\begin{align*}
\left|\frac{f(a)+f(b)}{2}-\frac{f'(b)-f'(a)}{b-a}\frac{(b-a)^2}{12}-\frac{1}{b-a}\int_{a}^{b}f(t)dt\right|
\leq\frac{(b-a)^{3/2}}{12 \sqrt{5}}\sqrt{\sigma(f'')}.
\end{align*}

$(3)$ $x=\frac{a+b}{2}$, we have
\begin{align*}
 \left|f\left(\frac{a+b}{2}\right)+\frac{f'(b)-f'(a)}{b-a}\frac{(b-a)^2}{24}-\frac{1}{b-a}\int_{a}^{b}f(t)dt\right|
 \leq\frac{(b-a)^{3/2}}{12 \sqrt{5}}\sqrt{\sigma(f'')}.
 \end{align*}
% It is obvious that \eqref{2.45} gives a smaller estimator than the above inequality.
\end{corollary}
\begin{corollary}
Let $f$ as in Theorem \ref{th2.3}. Additionally, if $f$ is symmetric about $x=\frac{a+b}{2}$,  then for all $x\in\left[a,\frac{a+b}{2}\right]$ we have
\begin{align*}
&\left|f(x)-\left(x-\frac{3a+b}{4}\right)f'(x)+\frac{f'(b)-f'(a)}{b-a}\left[\frac{1}{2}\left(x-\frac{3a+b}{4}\right)^2+\frac{(b-a)^2}{96}\right]-\frac{1}{b-a}\int_{a}^{b}f(t)dt\right|\nonumber\\
\leq&\frac{\sqrt{\sigma(f'')}}{b-a}\left\{\frac{1}{320}(a+b-2x)^5+\frac{1}{10}(x-a)^5-(b-a)\left[\frac{1}{2}\left(x-\frac{3a+b}{4}\right)^2+\frac{(b-a)^2}{96}\right]^2\right\}^{1/2}. \end{align*}
\end{corollary}

\section{Application to Composite Quadrature Rules}

Let $I_{n}: a=x_{0}<x_{1}<\cdot\cdot\cdot<x_{n-1}<x_{n}=b$ be a
partition of the interval $[a,b]$ and $h_{i}=x_{i+1}-x_{i}$
$(i=0,1,2,\cdot\cdot\cdot,n-1)$.

Consider the perturbed composite quadrature rules
\begin{equation}
Q_{n}^1(I_{n},f)=\frac{1}{2}\sum_{i=0}^{n-1}\left[f\left(\frac{3x_{i}+x_{i+1}}{4}\right)+f\left(\frac{x_{i}+3x_{i+1}}{4}\right)\right]h_{i}
+\sum_{i=0}^{n-1}\frac{f'(x_{i+1})-f'(x_{i})}{96}h_{i}^2\label{3.1}
\end{equation}
and
\begin{equation}
Q_{n}^2(I_{n},f)=\frac{1}{2}\sum_{i=0}^{n-1}\left[f\left(\frac{3x_{i}+x_{i+1}}{4}\right)+f\left(\frac{x_{i}+3x_{i+1}}{4}\right)\right]h_{i}
+\frac{\Gamma+\gamma}{192}\sum_{i=0}^{n-1}h_{i}^3. \label{3.2}
\end{equation}
The following result holds.
\begin{theorem}\label{th3.1}
Let $f:[a,b]\rightarrow\mathbb{R}$ be such that $f'$ is absolutely continuous on $[a,b]$. If $f''\in L^1[a,b]$ and $\gamma\leq f''(x)\leq\Gamma, \forall\ x\in [a,b]$, then for all $x\in\left[a,\frac{a+b}{2}\right]$ we have
\begin{equation*}
\int_{a}^{b}f(t)dt=Q_{n}^1(I_{n},f)+R_{n}^1(I_{n},f),
\end{equation*}
where $Q_{n}^1(I_{n},f)$ is defined by formula \eqref{3.1}, and the remainder $R_{n}^1(I_{n},f)$ satisfies the estimate
\begin{equation}\label{3.3}
|R_{n}^1(I_{n},f)|\leq\frac{\Gamma-\gamma}{128}\sum_{i=0}^{n-1}h_{i}^3.
\end{equation}
\end{theorem}
\begin{proof}
Applying inequality \eqref{2.8'} to the intervals $[x_{i},x_{i+1}]$, then we get
\begin{align*}
&\left|\int_{x_{i}}^{x_{i+1}}f(t)dt-\frac{1}{2}\left[f\left(\frac{3x_{i}+x_{i+1}}{4}\right)+f\left(\frac{x_{i}+3x_{i+1}}{4}\right)\right]h_{i}
-\frac{f'(x_{i+1})-f'(x_{i})}{96}h_{i}^2\right|\nonumber\\
\leq&\frac{\Gamma-\gamma}{128}  h_{i}^3
\end{align*}
for $i=0,1,2,\cdot\cdot\cdot,n-1.$
Now summing over $i$ from $0$ to $n-1$  and using the triangle
inequality,  we get \eqref{3.3}.
\end{proof}

\begin{theorem}\label{th3.2}
Let $f:[a,b]\rightarrow\mathbb{R}$ be such that $f'$ is absolutely continuous on $[a,b]$. If $f''\in L^1[a,b]$ and $\gamma\leq f''(x)\leq\Gamma, \forall\ x\in [a,b]$, then for all $x\in\left[a,\frac{a+b}{2}\right]$ we have
\begin{equation*}
\int_{a}^{b}f(t)dt=Q_{n}^2(I_{n},f)+R_{n}^2(I_{n},f),
\end{equation*}
where $Q_{n}^2(I_{n},f)$ is defined by formula \eqref{3.2}, and the remainder $R_{n}^2(I_{n},f)$ satisfies the estimate
\begin{equation}\label{3.4}
|R_{n}^2(I_{n},f)|\leq\frac{\Gamma-\gamma}{192}\sum_{i=0}^{n-1}h_{i}^3.
\end{equation}
\end{theorem}
\begin{proof}
Applying inequality \eqref{2.12'} to the intervals $[x_{i},x_{i+1}]$, then we get
\begin{align*}
&\left|\int_{x_{i}}^{x_{i+1}}f(t)dt-\frac{1}{2}\left[f\left(\frac{3x_{i}+x_{i+1}}{4}\right)+f\left(\frac{x_{i}+3x_{i+1}}{4}\right)\right]h_{i}
-\frac{\Gamma+\gamma}{192} h_{i}^3\right|\nonumber\\
\leq&\frac{\Gamma-\gamma}{192}  h_{i}^3
\end{align*}
for $i=0,1,2,\cdot\cdot\cdot,n-1.$
Now summing over $i$ from $0$ to $n-1$  and using the triangle
inequality,  we get \eqref{3.4}.
\end{proof}

\begin{theorem}\label{th3.3}
Let $f:[a,b]\rightarrow\mathbb{R}$ be such that $f'$ is absolutely continuous on $[a,b]$.  If $f''\in L^{1}[a,b]$ and $\gamma\leq f''(x)\leq\Gamma$, $\forall\ x\in[a,b]$, then for all $x\in\left[a,\frac{a+b}{2}\right]$ we have
\begin{equation*}
\int_{a}^{b}f(t)dt=Q_{n}^1(I_{n},f)+R_{n}^1(I_{n},f),
\end{equation*}
where $Q_{n}^1(I_{n},f)$ is defined by formula \eqref{3.1}, and the remainder $R_{n}^1(I_{n},f)$ satisfies the estimate
\begin{equation}\label{3.5}
|R_{n}^1(I_{n},f)|\leq\frac{1}{48}(S-\gamma)\sum_{i=0}^{n-1}h_{i}^3
\end{equation}
and
\begin{equation}\label{3.6}
|R_{n}^1(I_{n},f)|\leq\frac{1}{48}(S-\gamma)\sum_{i=0}^{n-1}h_{i}^3.
\end{equation}
\end{theorem}

\begin{proof}
Applying inequality \eqref{2.23} and \eqref{2.24} to the intervals $[x_{i},x_{i+1}]$, then we get
\begin{align*}
&\left|\int_{x_{i}}^{x_{i+1}}f(t)dt-\frac{1}{2}\left[f\left(\frac{3x_{i}+x_{i+1}}{4}\right)+f\left(\frac{x_{i}+3x_{i+1}}{4}\right)\right]h_{i}
-\frac{f'(x_{i+1})-f'(x_{i})}{96}h_{i}^2\right|\nonumber\\
\leq&\frac{1}{48}(S-\gamma)h_{i}^3
\end{align*}
and
\begin{align*}
&\left|\int_{x_{i}}^{x_{i+1}}f(t)dt-\frac{1}{2}\left[f\left(\frac{3x_{i}+x_{i+1}}{4}\right)+f\left(\frac{x_{i}+3x_{i+1}}{4}\right)\right]h_{i}
-\frac{f'(x_{i+1})-f'(x_{i})}{96}h_{i}^2\right|\nonumber\\
\leq&\frac{1}{48}(\Gamma-S)  h_{i}^3,
\end{align*}
for $i=0,1,2,\cdot\cdot\cdot,n-1.$
Now summing over $i$ from $0$ to $n-1$  and using the triangle
inequality,  we get \eqref{3.5} and \eqref{3.6}.
\end{proof}

\begin{theorem}\label{th3.4}
Let $f:[a,b]\rightarrow\mathbb{R}$ be a thrice continuously differentiable mapping in $(a,b)$ with $f'''\in L^{2}[a,b]$. Then for all $x\in\left[a,\frac{a+b}{2}\right]$ we have
\begin{equation*}
\int_{a}^{b}f(t)dt=Q_{n}^1(I_{n},f)+R_{n}^1(I_{n},f),
\end{equation*}
where $Q_{n}^1(I_{n},f)$ is defined by formula \eqref{3.1}, and the remainder $R_{n}^1(I_{n},f)$ satisfies the estimate
\begin{equation}\label{3.7}
|R_{n}^1(I_{n},f)|\leq\frac{\|f'''\|_{2}}{48\pi\sqrt{5}}\sum_{i=0}^{n-1}h_{i}^{7/2}.
\end{equation}
\end{theorem}

\begin{proof}
Applying inequality \eqref{2.33}  to the intervals $[x_{i},x_{i+1}]$, then we get
\begin{align*}
&\left|\int_{x_{i}}^{x_{i+1}}f(t)dt-\frac{1}{2}\left[f\left(\frac{3x_{i}+x_{i+1}}{4}\right)+f\left(\frac{x_{i}+3x_{i+1}}{4}\right)\right]h_{i}
-\frac{f'(x_{i+1})-f'(x_{i})}{96}h_{i}^2\right|\nonumber\\
\leq&\frac{h_{i}^{7/2}}{48\pi\sqrt{5}}\|f'''\|_{2},
\end{align*}
for $i=0,1,2,\cdot\cdot\cdot,n-1.$
Now summing over $i$ from $0$ to $n-1$  and using the triangle
inequality,  we get \eqref{3.7}.
\end{proof}

\begin{theorem}\label{th3.5}
Let $f:[a,b]\rightarrow\mathbb{R}$ be such that $f'$ is absolutely continuous on $[a,b]$ with $f''\in L^{2}[a,b]$. Then for all $x\in\left[a,\frac{a+b}{2}\right]$ we have
\begin{equation*}
\int_{a}^{b}f(t)dt=Q_{n}^1(I_{n},f)+R_{n}^1(I_{n},f),
\end{equation*}
where $Q_{n}^1(I_{n},f)$ is defined by formula \eqref{3.1}, and the remainder $R_{n}^1(I_{n},f)$ satisfies the estimate
\begin{equation}\label{3.8}
|R_{n}^1(I_{n},f)|\leq\frac{\sqrt{\sigma(f'')}}{48\sqrt{5}}\sum_{i=0}^{n-1}h_{i}^{5/2}.
\end{equation}
\end{theorem}

\begin{proof}
Applying inequality \eqref{2.38}  to the intervals $[x_{i},x_{i+1}]$, then we get
\begin{align*}
&\left|\int_{x_{i}}^{x_{i+1}}f(t)dt-\frac{1}{2}\left[f\left(\frac{3x_{i}+x_{i+1}}{4}\right)+f\left(\frac{x_{i}+3x_{i+1}}{4}\right)\right]h_{i}
-\frac{f'(x_{i+1})-f'(x_{i})}{96}h_{i}^2\right|\nonumber\\
\leq&\frac{h_{i}^{5/2}}{48\sqrt{5}}\sqrt{\sigma(f'')},
\end{align*}
for $i=0,1,2,\cdot\cdot\cdot,n-1.$
Now summing over $i$ from $0$ to $n-1$  and using the triangle
inequality,  we get \eqref{3.8}.
\end{proof}

%
%\begin{table}[h] % 制表
%\caption{Numerical results}
%\centering
%\begin{tabular}{crrrrr}
%\hline\noalign{\smallskip}
%$f(x)$ & $n$ & $[a,b]$ & $\int_{a}^{b}f(x)dx$ & $Q_{n}^{1}$ & $Q_{n}^{2}$ \\
%  $\cos x-x$ & 20 & $\left[0,\frac{\pi}{2}\right]$ & -0.233701 &  -0.233636 &  -0.233636 \\
%  $e^{2x}\cos(e^{x})$ & 20 & $[0,1]$ & -1.176887 & -1.176316 & -1.176316 \\
%  $\frac{1}{x^4+4x^2+3}$ & 20 & $[0,1]$ & 0.241549 & 0.241554 & 0.241554 \\
%  $\tan x+x$ & 20 & $\left[0,\frac{\pi}{4}\right]$ & 0.654999 & 0.654983 & 0.654983 \\
%  $\ln(x^2+1)$ & 20 & $[-1,1]$ & 0.527887 & 0.527679 & 0.527679 \\
%\end{tabular}
%\end{table}
In order to compare the error between these two methods and composite trapezoidal formula, we apply the above methods and composite trapezoidal formula to specific examples. For this purpose consider the following Table:

\begin{table}[h] % 制表
\caption{Numerical results}
\centering
\begin{tabular}{crrrrrrrr}
\hline\noalign{\smallskip}
$f(x)$ & $n$ & $[a, b]$ & $\int_{a}^{b}f(x)dx$ & $T_{n}$ & Error of $T_{n}$ & $Q_{n}^{1}$ & $Q_{n}^{2}$ &Error of $Q_{n}^{1}$and $Q_{n}^{2}$ \\
\noalign{\smallskip}\hline\noalign{\smallskip}
  $\cos x-x$ & 20 & $\left[0, \frac{\pi}{2}\right]$ & -0.233701 & -0.234215 & 5.14E-4 &  -0.233636 &  -0.233636 & 6.5E-5\\
  $e^{2x}\cos(e^{x})$ & 20 & $[0, 1]$ & -1.176887 & -1.181466 & 4.579E-3 & -1.176316 & -1.176316 & 5.71E-4 \\
  $\frac{1}{x^4+4x^2+3}$ & 10 & $[0, 1]$ & 0.241549 & 0.241393 & 1.56E-4 & 0.241569 & 0.241569 & 2E-5 \\
  $\tan x+x$ & 20 & $\left[0, \frac{\pi}{4}\right]$ & 0.654999 &0.655127 & 1.28E-4 &0.654983 & 0.654983 & 7E-6 \\
  $\ln(x^2+1)$ & 20 & $[-1, 1]$ & 0.527887 & 0.529554 & 1.667E-3 & 0.527679 &0.527679 & 2.08E-4\\
\noalign{\smallskip}\hline
\end{tabular}
\end{table}

\section{Application to probability density functions}

Now, let $X$ be a random variable taking values in the finite interval
$[a,b]$, with the probability density function $f : [a, b]\rightarrow [0, 1]$ and with
the cumulative distribution function $$F (x) = Pr (X \leq  x) = \int_a^x f (t) dt.$$

The following results hold:
\begin{theorem}\label{th4.1}
With the assumptions of Theorem \ref{th2.1}, we have
\begin{align}\label{4.1}
&\left|\frac{1}{2}[F(x)+F(a+b-x)]-\left(x-\frac{3a+b}{4}\right)\frac{f'(x)-f'(a+b-x)}{2}\right.\nonumber\\
&\left.+\frac{f'(b)-f'(a)}{b-a}\left[\frac{1}{2}\left(x-\frac{3a+b}{4}\right)^2+\frac{(b-a)^2}{96}\right]-\frac{b-E(X)}{b-a}\right|\nonumber\\
\leq&\frac{1}{8}(\Gamma-\gamma)\left[\frac{b-a}{4}+\left|x-\frac{3a+b}{4}\right|\right]^2.
\end{align}
for all $x\in[a,\frac{a+b}{2}]$, where $E (X)$ is the expectation of $X$.
\end{theorem}
\begin{proof}  By \eqref{2.4}  on choosing $f = F$ and taking into account
$$E(X)=\int_a^b t dF(t)=b-\int_a^b F(t)dt,$$
we obtain \eqref{4.1}.
\end{proof}

\begin{corollary}
Under the  assumptions of Theorem \ref{th4.1} with $x=\frac{3a+b}{4}$, we have
\begin{align}
 \left|\frac{1}{2}\left[F\left(\frac{3a+b}{4}\right)+F\left(\frac{a+3b}{4}\right)\right]+\frac{b-a}{96}[f'(b)-f'(a)]-\frac{b-E(x)}{b-a}\right|
\leq \frac{\Gamma-\gamma}{128}(b-a)^2.
\end{align}
\end{corollary}

\begin{theorem}\label{th4.2}
With the assumptions of Theorem \ref{th2.2}, we have
\begin{align}\label{4.3}
&\left|\frac{1}{2}[F(x)+F(a+b-x)]-\left(x-\frac{3a+b}{4}\right)\frac{f(x)-f(a+b-x)}{2}\right.\nonumber\\
&\left.+\frac{\Gamma+\gamma}{2}\left[\frac{1}{2}\left(x-\frac{3a+b}{4}\right)^2+\frac{(b-a)^2}{96}\right]-\frac{b-E(X)}{b-a}\right|\nonumber\\
\leq&\frac{\Gamma-\gamma}{2}\left[\frac{1}{2}\left(x-\frac{3a+b}{4}\right)^2+\frac{(b-a)^2}{96}\right]
\end{align}
for all $x\in[a,\frac{a+b}{2}]$, where $E (X)$ is the expectation of $X$.
\end{theorem}
\begin{proof}  By \eqref{2.8}  on choosing $f = F$ and taking into account
$$E(X)=\int_a^b t dF(t)=b-\int_a^b F(t)dt,$$
we obtain \eqref{4.3}.
\end{proof}

\begin{corollary}
Under the  assumptions of Theorem \ref{th4.2} with $x=\frac{3a+b}{4}$, we have
\begin{align}
 \left|\frac{1}{2}\left[F\left(\frac{3a+b}{4}\right)+F\left(\frac{a+3b}{4}\right)\right]+\frac{\Gamma+\gamma}{192}(b-a)^2-\frac{b-E(x)}{b-a}\right|
\leq \frac{\Gamma-\gamma}{192}(b-a)^2.
\end{align}
\end{corollary}

\begin{theorem}\label{th4.3}
With the assumptions of Theorem \ref{th2.3}, we have
\begin{align}\label{4.5}
&\left|\frac{1}{2}[F(x)+F(a+b-x)]-\left(x-\frac{3a+b}{4}\right)\frac{f'(x)-f'(a+b-x)}{2}\right.\nonumber\\
&\left.+\frac{f'(b)-f'(a)}{b-a}\left[\frac{1}{2}\left(x-\frac{3a+b}{4}\right)^2+\frac{(b-a)^2}{96}\right]-\frac{b-E(X)}{b-a}\right|\nonumber\\
\leq & (S-\gamma)\left[\frac{(b-a)^2}{48}+ \frac{b-a}{4}\left|x-\frac{3a+b}{4}\right|\right]
\end{align}
and
\begin{align}\label{4.6}
&\left|\frac{1}{2}[F(x)+F(a+b-x)]-\left(x-\frac{3a+b}{4}\right)\frac{f'(x)-f'(a+b-x)}{2}\right.\nonumber\\
&\left.+\frac{f'(b)-f'(a)}{b-a}\left[\frac{1}{2}\left(x-\frac{3a+b}{4}\right)^2+\frac{(b-a)^2}{96}\right]-\frac{b-E(X)}{b-a}\right|\nonumber\\
\leq & (\Gamma-S)\left[\frac{(b-a)^2}{48}+ \frac{b-a}{4}\left|x-\frac{3a+b}{4}\right|\right]
\end{align}
for all $x\in[a,\frac{a+b}{2}]$, where $E (X)$ is the expectation of $X$.
\end{theorem}
\begin{proof}  By \eqref{2.14} and \eqref{2.15}  on choosing $f = F$ and taking into account
$$E(X)=\int_a^b t dF(t)=b-\int_a^b F(t)dt,$$
we obtain \eqref{4.5} and \eqref{4.6}.
\end{proof}

\begin{corollary}
Under the  assumptions of Theorem \ref{th4.3} with $x=\frac{3a+b}{4}$, we have
\begin{align}
 \left|\frac{1}{2}\left[F\left(\frac{3a+b}{4}\right)+F\left(\frac{a+3b}{4}\right)\right]+\frac{b-a}{96}[f'(b)-f'(a)]-\frac{b-E(x)}{b-a}\right|
\leq \frac{1}{48}(S-\gamma)(b-a)^2
\end{align}
and
\begin{align}
 \left|\frac{1}{2}\left[F\left(\frac{3a+b}{4}\right)+F\left(\frac{a+3b}{4}\right)\right]+\frac{b-a}{96}[f'(b)-f'(a)]-\frac{b-E(x)}{b-a}\right|
\leq \frac{1}{48}(\Gamma-S)(b-a)^2.
\end{align}
\end{corollary}

\begin{theorem}\label{th4.4}
With the assumptions of Theorem \ref{th2.4}, we have
\begin{align}\label{4.7}
&\left|\frac{1}{2}[F(x)+F(a+b-x)]-\left(x-\frac{3a+b}{4}\right)\frac{f'(x)-f'(a+b-x)}{2}\right.\nonumber\\
&\left.+\frac{f'(b)-f'(a)}{b-a}\left[\frac{1}{2}\left(x-\frac{3a+b}{4}\right)^2+\frac{(b-a)^2}{96}\right]-\frac{b-E(X)}{b-a}\right|\nonumber\\
\leq&\frac{1}{\pi}\|f'''\|_{2}\left\{\frac{1}{320}(a+b-2x)^5+\frac{1}{10}(x-a)^5-(b-a)\left[\frac{1}{2}\left(x-\frac{3a+b}{4}\right)^2+\frac{(b-a)^2}{96}
\right]^2\right\}^{1/2}
\end{align}
for all $x\in[a,\frac{a+b}{2}]$, where $E (X)$ is the expectation of $X$.
\end{theorem}
\begin{proof}  By \eqref{2.29}  on choosing $f = F$ and taking into account
$$E(X)=\int_a^b t dF(t)=b-\int_a^b F(t)dt,$$
we obtain \eqref{4.7}.
\end{proof}

\begin{corollary}
Under the  assumptions of Theorem \ref{th4.4} with $x=\frac{3a+b}{4}$, we have
\begin{align}
 \left|\frac{1}{2}\left[F\left(\frac{3a+b}{4}\right)+F\left(\frac{a+3b}{4}\right)\right]+\frac{b-a}{96}[f'(b)-f'(a)]-\frac{b-E(x)}{b-a}\right|
\leq\frac{(b-a)^{5/2}}{48\pi\sqrt{5}}\|f'''\|_{2}.
\end{align}
\end{corollary}

\begin{theorem}\label{th4.5}
With the assumptions of Theorem \ref{th2.5}, we have
\begin{align}\label{4.11}
&\left|\frac{1}{2}[F(x)+F(a+b-x)]-\left(x-\frac{3a+b}{4}\right)\frac{f'(x)-f'(a+b-x)}{2}\right.\nonumber\\
&\left.+\frac{f'(b)-f'(a)}{b-a}\left[\frac{1}{2}\left(x-\frac{3a+b}{4}\right)^2+\frac{(b-a)^2}{96}\right]-\frac{b-E(X)}{b-a}\right|\nonumber\\
\leq&\frac{\sqrt{\sigma(f'')}}{b-a}\left\{\frac{1}{320}(a+b-2x)^5+\frac{1}{10}(x-a)^5-(b-a)\left[\frac{1}{2}\left(x-\frac{3a+b}{4}\right)^2
+\frac{(b-a)^2}{96}\right]^2\right\}^{1/2}
\end{align}
for all $x\in[a,\frac{a+b}{2}]$, where $E (X)$ is the expectation of $X$.
\end{theorem}
\begin{proof}  By \eqref{2.36}  on choosing $f = F$ and taking into account
$$E(X)=\int_a^b t dF(t)=b-\int_a^b F(t)dt,$$
we obtain \eqref{4.11}.
\end{proof}

\begin{corollary}
Under the  assumptions of Theorem \ref{th4.5} with $x=\frac{3a+b}{4}$, we have
\begin{align}
 \left|\frac{1}{2}\left[F\left(\frac{3a+b}{4}\right)+F\left(\frac{a+3b}{4}\right)\right]+\frac{b-a}{96}[f'(b)-f'(a)]-\frac{b-E(x)}{b-a}\right|
\leq\frac{(b-a)^{3/2}}{48\sqrt{5}}\sqrt{\sigma(f'')}.
\end{align}
\end{corollary}

\subsection*{Acknowledgments}
This work was partly supported by the National Natural Science Foundation
of China (Grant No. 40975002) and the Natural Science Foundation of the Jiangsu
Higher Education Institutions (Grant No. 09KJB110005).


\begin{thebibliography}{00}

%\bibitem{a2011}
%M. W. Alomari, A companion of Dragomir's generalization of Ostrowski's inequality and applications in numerical
%integration,  RGMIA Res. Rep. Coll., 14 (2011) article 50.

\bibitem{a20111}
M. W. Alomari, A companion of Ostrowski's inequality with applications, Transylv. J. Math. Mech. {\bf 3} (2011), no.~1, 9--14.



\bibitem{a20112}
M. W. Alomari,
A companion of Ostrowski's inequality  for mappings whose first derivatives are bounded and applications in numerical integration, RGMIA Res. Rep. Coll., 14 (2011) article 57.

\bibitem{a20113}
M. W. Alomari,
A generalization of companion inequality of  Ostrowski's type for mappings whose first derivatives are bounded and applications in numerical integration, RGMIA Res. Rep. Coll., 14 (2011) article 58.




\bibitem{bdg2009}
N. S. Barnett, S. S. Dragomir\ and\ I. Gomm, A companion for the Ostrowski and the generalised trapezoid inequalities, Math. Comput. Modelling {\bf 50} (2009), no.~1-2, 179--187.

\bibitem{cdr1999}
P. Cerone, S. S. Dragomir, J.  Roumeliotis,  An inequality of Ostrowski type for mappings whose second derivatives are bounded and applications, East Asian Math. J,  15 (1999), no 1--9.

\bibitem{d2002}
S. S. Dragomir, A companions of Ostrowski's inequality for functions of bounded variation and applications, RGMIA Res. Rep. Coll., 5 (2002), Supp., article 28.

\bibitem{d20051}
S. S. Dragomir, Ostrowski type inequalities for functions defined on linear spaces and applications for semi-inner products, J. Concr. Appl. Math. {\bf 3} (2005), no.~1, 91--103.

\bibitem{d2005}
S. S. Dragomir, Some companions of Ostrowski's inequality for absolutely continuous functions and applications, Bull. Korean Math. Soc. {\bf 42} (2005), no.~2, 213--230.

\bibitem{d2011}
S. S. Dragomir, Ostrowski's type inequalities for continuous functions of selfadjoint operators on Hilbert spaces: a survey of recent results, Ann. Funct. Anal. {\bf 2} (2011), no.~1, 139--205.


\bibitem{ds2000}
S. S. Dragomir\ and\ A. Sofo, An integral inequality for twice differentiable mappings and applications, Tamkang J. Math. {\bf 31} (2000), no.~4, 257--266.

\bibitem{d2001}
J. Duoandikoetxea, A unified approach to several inequalities involving functions and derivatives, Czechoslovak Math. J. {\bf 51(126)} (2001), no.~2, 363--376.

\bibitem{g1935}
G. Gr\"uss, \"Uber das Maximum des absoluten Betrages von $\frac{1}{{b - a}}\int\limits\sb a\sp b {f\left( x \right)} g\left( x \right)dx - \frac{1}{{\left( {b - a} \right)\sp 2 }}\int\limits\sb a\sp b {f\left( x \right)dx} \int\limits\sb a\sp b g \left( x \right)dx$, Math. Z. {\bf 39} (1935), no.~1, 215--226.

\bibitem{gs2002}
A. Guessab\ and\ G. Schmeisser, Sharp integral inequalities of the Hermite-Hadamard type, J. Approx. Theory {\bf 115} (2002), no.~2, 260--288.

\bibitem{hn2011}
V. N. Huy\ and\ Q. -A. Ng\^o, New bounds for the Ostrowski-like type inequalities, Bull. Korean Math. Soc. {\bf 48} (2011), no.~1, 95--104.

\bibitem{l2008}
W. J. Liu,  Several error inequalities for a quadrature formula with a parameter and applications, Comput. Math. Appl. {\bf 56} (2008), no.~7, 1766--1772.

\bibitem{l2010}
W. J. Liu, Some weighted integral inequalities with a parameter and applications, Acta Appl. Math. {\bf 109} (2010), no.~2, 389--400.


\bibitem{lxw2010}
W. J. Liu, Q. L. Xue\ and\ S. F. Wang, New generalization of perturbed Ostrowski type inequalities and applications, J. Appl. Math. Comput. {\bf 32} (2010), no.~1, 157--169.

\bibitem{l2007}
Z. Liu, Note on a paper by N. Ujevi\'c, Appl. Math. Lett. {\bf 20} (2007), no.~6, 659--663.

\bibitem{liu2009}
Z. Liu, Some companions of an Ostrowski type inequality and applications, JIPAM. J. Inequal. Pure Appl. Math. {\bf 10} (2009), no.~2, Article 52, 12 pp.

\bibitem{mpf1991}
D. S. Mitrinovi\'c, J. E. Pe\v cari\'c\ and\ A. M. Fink, {\it Inequalities involving functions and their integrals and derivatives}, Mathematics and its Applications (East European Series), 53, Kluwer Acad. Publ., Dordrecht, 1991.


\bibitem{o1938}
A. Ostrowski, \"Uber die Absolutabweichung einer differentiierbaren Funktion von ihrem Integralmittelwert, Comment. Math. Helv. {\bf 10} (1937), no.~1, 226--227.

\bibitem{s20101}
M. Z. Sarikaya, On the Ostrowski type integral inequality, Acta Math. Univ. Comenian. (N.S.) {\bf 79} (2010), no.~1, 129--134.

\bibitem{s2010}
M. Z. Sarikaya, New weighted Ostrowski and \v Ceby\v sev type inequalities on time scales, Comput. Math. Appl. {\bf 60} (2010), no.~5, 1510--1514.

\bibitem{ss2011}
E. Set\ and\ M. Z. Sar\i kaya, On the generalization of Ostrowski and Gr\"uss type discrete inequalities, Comput. Math. Appl. {\bf 62} (2011), no.~1, 455--461.

\bibitem{thd2008}
K.-L. Tseng, S.-R. Hwang\ and\ S. S. Dragomir, Generalizations of weighted Ostrowski type inequalities for mappings of bounded variation and their applications, Comput. Math. Appl. {\bf 55} (2008), no.~8, 1785--1793.

\bibitem{thyc2011}
K.-L. Tseng,  S.-R. Hwang, G.-S. Yang\ and\ Y.-M. Chou, Weighted Ostrowski integral inequality for mappings of bounded variation, Taiwanese J. Math. {\bf 15} (2011), no.~2, 573--585.

\bibitem{u2003}
N. Ujevi\'c, New bounds for the first inequality of Ostrowski-Gr\"uss type and applications, Comput. Math. Appl. {\bf 46} (2003), no.~2-3, 421--427.

\bibitem{v2011}
S. W. Vong, A note on some Ostrowski-like type inequalities, Comput. Math. Appl. {\bf 62} (2011), no.~1, 532--535.



















\end{thebibliography}
\end{document}